\documentclass[11pt,reqno]{amsart} 
\usepackage{amssymb}
\usepackage{amsmath}
\usepackage{amsfonts}
\usepackage[english]{babel}
\usepackage[T1]{fontenc}
\usepackage[latin1]{inputenc}
\usepackage{amsthm}
\usepackage[all]{xy}

\theoremstyle{plain}
\newtheorem{theorem}{Theorem}
\newtheorem{proposition}[theorem]{Proposition}
\newtheorem{corollary}[theorem]{Corollary}
\newtheorem{lemma}[theorem]{Lemma}

\theoremstyle{remark}

\newtheorem{example}[theorem]{Example}
\newtheorem{question}[theorem]{Question}

\newcommand{\e}{{\mathrm e}}

\newcommand{\UnitDisk}{\mathbb{D}}

\newcommand{\UnitCircle}{\mathbb{T}}

\title[Survey of composition operators]{Composition operators on vector-valued analytic function spaces: a survey}
\author{Jussi Laitila}
\address{Jussi Laitila, Department of Biosciences, P.O.\ Box 65, FI-00014 University of Helsinki, Helsinki, Finland}
\email{jussi.laitila@helsinki.fi}
\author{Hans-Olav Tylli}
\address{Hans-Olav Tylli, Department of Mathematics and Statistics, P.O.\ Box 68, FI-00014 University of Helsinki, Helsinki, Finland}
\email{hans-olav.tylli@helsinki.fi}

\begin{document}

\begin{abstract}
We survey recent results about composition operators induced by analytic self-maps 
of the unit disk in the complex plane  on various Banach spaces of
analytic functions taking values in infinite-dimensional Banach spaces. We
mostly concentrate on the research line into qualitative properties such as weak
compactness, initiated by Liu, Saksman and Tylli (1998), and continued in
several other papers. 
We discuss composition operators on
strong, respectively weak, spaces of vector-valued analytic functions, as
well as between weak and strong spaces. As concrete examples, 
we review more carefully and present some new
observations in the cases of vector-valued Hardy and $BMOA$ spaces,
though the study of composition operators 
has been extended to a wide range of spaces of vector-valued 
analytic functions, including spaces defined on other domains.
Several open problems are stated.
\end{abstract}

\date{20.3.2014}

\maketitle

\section{Introduction}

Let $\mathbb D = \{z \in \mathbb C\colon \vert z\vert < 1\}$ be the open unit disk in the complex plane
$\mathbb C$,
and let $\varphi\colon \mathbb D \to \mathbb D$ be a fixed analytic map. 
The classical study  of the analytic composition operators 
$C_\varphi$, where 
\[
f \mapsto C_\varphi(f) =  f \circ \varphi,
\] 
originates from the work of Ryff (1966) and Nordgren (1968). For instance, they observed that
any $C_\varphi$ defines a  bounded operator $H^p \to H^p$
as a consequence of the Littlewood subordination principle. 
Recall that for $1 \le p < \infty$, the analytic function $f\colon \mathbb D \to \mathbb C$
belongs to the Hardy space $H^p$ if 
 \[
 \Vert f\Vert^p_{H^p} = \sup_{0 \le r < 1} \int_{\mathbb  T}  \vert f(r\xi)\vert^p dm(\xi)
 < \infty
 \]
where $\mathbb T = \partial \mathbb D = [0,2\pi]$ and $dm(\e^{it}) = dt/2\pi$. 
The space $H^\infty$ consists of the bounded analytic functions.
Subsequently an extensive literature 
has emerged, where a very wide variety of 
properties  of analytic composition operators  has been addressed 
on a large number of  spaces of analytic functions. 
We refer to \cite{Sh93} and \cite{CMC95} for comprehensive accounts 
of the theory until ca.\ 1995.

This survey reviews more recent results about 
composition operators  on  various Banach spaces of vec\-tor-valued analytic functions
including the vec\-tor-valued Hardy space  $H^p(X)$, 
where $X$ is a complex Banach space. 
Let $f\colon \mathbb D \to X$ be a vector-valued analytic function and let $1\le p < \infty$. Then $f \in H^p(X)$, if 
 \[
 \Vert f\Vert^p_{H^p(X)} = \sup_{0 \le r < 1} \int_{\mathbb T}  \Vert f(r\xi)\Vert^p_X dm(\xi) < \infty. 
 \]
Moreover, $f\in H^\infty(X)$, if $\Vert f\Vert_{H^\infty}=\sup_{z\in\UnitDisk}\Vert f(z) \Vert_X <\infty$.
In this notation $H^p = H^p(\mathbb C)$.
 Above the analyticity of $f\colon \mathbb D \to X$ means that   the scalar-valued function $x^* \circ f$ 
 is analytic $\mathbb D \to \mathbb C$ for any functional $x^* \in X^*$  (that is, 
 $f$ is weakly analytic). 
This is equivalent to the requirement that the $X$-valued derivative $f'(z)$
exists for all points $z\in \UnitDisk$ (that is, $f$ is strongly analytic). For the basics of vector-valued
analytic functions, see for example \cite{HilleP}.

 Qualitative properties of the vector-valued composition operators 
 $f \mapsto f \circ \varphi$ on $H^p(X)$ and certain other spaces were first systematically 
studied by Liu, Saksman and Tylli \cite{LST98}.  Independently 
Hornor and Jamison \cite{HJ99}  considered the operators  $f \mapsto f \circ \varphi$  on 
$H^p(X)$ with different aims, and Sharma and Bhanu \cite{SB99} looked at some of their basic
operator properties on $H^2(X)$, where $X$ is a Hilbert space.

We mostly concentrate on qualitative properties, such as weak compactness,
of composition operators on several  Banach spaces of vector-valued 
analytic functions of both strong and weak type defined on $\UnitDisk$. Weak type spaces were 
introduced into this context by Bonet, Domanski and Lindstr\"om \cite{BDL01},
and in this case the techniques differ from those of the strong type spaces.  
In section \ref{frame} we introduce a general framework for vector-valued composition
operators in order to provide a convenient general perspective into the study,
and we review results that illustrate both similarities 
and differences compared to the scalar-valued case $X = \mathbb C$. We also highlight new phenomena that do not have any counterparts for scalar composition operators. For instance, composition operators can be studied between a weak and a strong space.
In the final section we briefly discuss attempts to generalize the larger class of
weighted composition operators to the vector-valued setting.
Some vector-valued arguments are sketched, but we mostly assume that the 
basic scalar theory is known from  \cite{Sh93} and \cite{CMC95}.
 
Composition operators of different nature occur in various other settings.
For instance, there is a well-developed theory of 
the  composition operators $S \mapsto A \circ S \circ B$, where $A$ and $B$ 
are fixed bounded operators, on spaces of linear operators, 
see e.g. the survey \cite{ST06}. 
Properties of  such composition operators 
will actually be required  in section \ref{weak} below. 
 
\section{a general framework}\label{frame}
 
We first introduce a flexible general framework for the study of 
qualitative properties of vector-valued composition operators, which
will facilitate a discussion of  some common features.

Suppose that $A$ is a Banach space of analytic functions $\mathbb{D} \to \mathbb{C}$
and let $A(X)$ be an associated vector-valued  Banach space of 
analytic functions $\mathbb{D} \to X$,  where $X$ is a complex Banach space.
Assume that the following properties hold for the pair $(A,A(X))$
for all Banach spaces $X$:

\begin{itemize}
\item[(AF1)] The constant maps $f(z) \equiv c$ belong to $A$ for all $c \in \mathbb C$.

\item[(AF2)]  $f \mapsto f \otimes x$ defines a bounded linear operator 
$J_x\colon A \to A(X)$ 
for  any $x \in X$, where $(f \otimes x)(z) = f(z)x$ for $z \in \mathbb{D}$.

\item[(AF3)]  $g \mapsto x^* \circ g$ defines a bounded linear operator $Q_{x^*}\colon A(X) \to A$ for any $x^* \in X^*$.

\item[(AF4)] The point evaluations $\delta_z$,  where $\delta_z(f) = f(z)$
for $f \in A(X)$, are bounded $A(X) \to X$ for all  $z \in \mathbb D$. 

\end{itemize}

It follows from (AF1) and (AF2) that the vector-valued constant maps $z \mapsto f_x(z) \equiv  x$, 
that is $f_x = 1 \otimes x$,  belong to $A(X)$ for all $x \in X$. It is easy to check that
(AF1) -- (AF4) are satisfied for the pair $(H^p, H^p(X))$ for any $X$.
Note that  for Banach spaces of analytic functions defined on other 
domains, such as a half-plane or the plane 
$\mathbb C$,  condition (AF1) may not be relevant and the above framework cannot be 
applied in this form.

Suppose that $A$ and $B$ are Banach spaces of analytic functions 
$\mathbb{D} \to \mathbb{C}$ so that $(A,A(X))$ and $(B,B(X))$
satisfy  (AF1) -- (AF4),
where $A(X)$ and $B(X)$ are $X$-valued Banach spaces of analytic functions on
$\mathbb D$ associated with $A$, respectively $B$. 
Let $\varphi\colon \mathbb{D} \to \mathbb{D}$ be a given analytic self-map, and suppose that
the vector-valued composition operator $\widetilde{C_\varphi}$ is bounded $A(X) \to B(X)$,
where  $f \mapsto \widetilde{C_\varphi}(f) = f \circ \varphi$. In order to distinguish between
composition operators acting on different spaces we will in the sequel  use 
$C_\varphi\colon A \to B$ for the composition operator 
$f \mapsto f \circ \varphi$ in the scalar-valued setting, that is, in the case $X = \mathbb C$, and 
$\widetilde{C_\varphi}$ for its vector-valued version $A(X) \to B(X)$.

The following general formulation is partly motivated by \cite[Prop.\ 1]{BDL01}.

\begin{proposition}\label{GF}
The following factorizations hold.
\begin{itemize}
\item[(F1)] Let  $x \in X$, $x^* \in X^*$ be norm-$1$ vectors so that $\langle x^*,x\rangle = 1$.
Then
\[
\xymatrix{
A(X) \ar[r]^{\widetilde{C_\varphi}} & B(X) \ar[d]^{Q_{x^*}}\\
A \ar[u]^{J_x} \ar[r]^{C_\varphi} & B}
\]
commutes.

\item[(F2)]  Let $j(x) = f_x$ for $x \in X$, where $f_x(z) \equiv x$ for all $z \in \mathbb D$. Then
\[
\xymatrix{
A(X) \ar[r]^{\widetilde{C_\varphi}}  & B(X) \ar[d]^{\delta_0}\\
X \ar[u]^j \ar[r]^{I_X} & X}
\]
commutes, where $I_X$ is the identity operator on $X$.
\end{itemize}
\end{proposition}

\begin{proof}
Note towards  (F1) that $x^*(\widetilde{C_\varphi}(f \otimes x)) = x^*((f \circ \varphi) \otimes x)
= C_\varphi(f)$ for $f \in A$,  while 
$\delta_0(\widetilde{C_\varphi}(f_x)) = \delta_0(f_x) = x$ for $x \in X$.
\end{proof}

The above factorizations place some inherent restrictions 
on possible qualitative properties of the vector-valued operators $\widetilde{C_\varphi}$.
 Roughly speaking,
part (3) below states that  $\widetilde{C_\varphi}\colon A(X) \to B(X)$
cannot have any qualitative properties inherited under composition of linear operators
 that are not shared by $C_\varphi\colon  A \to B$ and the identity operator $I_X\colon X \to X$. 
Thus  Banach space properties of $X$ also influence (qualitative) properties of 
$\widetilde{C_\varphi}$.

\begin{corollary}\label{qcor}
Let $X$ be a complex Banach space.
\begin{enumerate}
\item
If $\widetilde{C_\varphi}$ is bounded $A(X) \to B(X)$, then $C_\varphi$ is bounded
$A \to B$.
\item
If  $\widetilde{C_\varphi}\colon A(X) \to B(X)$ is compact, then $C_\varphi$ is compact
$A \to B$ and $X$ is finite-dimensional. In particular, if
$X$ is infinite-dimensional, then $\widetilde{C_\varphi}$ is 
never compact $A(X) \to B(X)$.
\item
Let $\mathcal I$ be an operator ideal in the sense of Pietsch \cite{Pi}. If
$\widetilde{C_\varphi}\colon A(X) \to B(X)$ belongs to $\mathcal I$, then $I_X$ as well as 
$C_\varphi\colon A \to B$ belong to $\mathcal I$.
\end{enumerate}
\end{corollary}

 In fact, by (F1) and (F2) the compactness of  $\widetilde{C_\varphi}\colon A(X) \to B(X)$ 
 implies that both  $C_\varphi\colon  A \to B$ and $I_X$ are compact, that is, 
$X$ is finite-dimensional. Part (3) is verified in a similar fashion. For a converse of (2), see Proposition \ref{comp}.
 
\section{Weak compactness on $H^1(X)$ and other vector-valued spaces}\label{strongs}

Let $\varphi\colon \mathbb D \to \mathbb D$ be any analytic map. It was observed independently
in \cite{LST98} and \cite{HJ99} that  $\widetilde{C_\varphi}$ is bounded on $H^p(X)$, 
while \cite{SB99}  contains the case $H^2(X)$, where $X$ is a Hilbert space.
Boundedness can be verified in the following manner by a small modification of an argument for scalar $H^p$ spaces.
Note first that $z \mapsto \Vert f(z)\Vert_X$ is a subharmonic map on $\mathbb D$
for any analytic function $f\colon \mathbb D \to  X$, since
\[
\Vert f(z)\Vert_X = \sup_{\Vert x^* \Vert \le 1} \vert \langle x^*,f(z) \rangle \vert, 
\quad z \in \mathbb D.
\]
 Consequently, if $\varphi(0) = 0$, then the 
Littlewood inequality \cite[Thm.\ 2.22]{CMC95}
yields that $\Vert \widetilde{C_\varphi}(f)\Vert_{H^p(X)}
\le  \Vert f\Vert^p_{H^p(X)}$ for $f \in H^p(X)$. 

For the general case let $\sigma_a\colon \mathbb D \to \mathbb D$ be the M\"obius transformation defined
by $\sigma_a(z) = \frac{a-z}{1-\overline{a}z}$ for $a \in \mathbb D$.
If $\varphi(0) \neq  0$ , let  $\psi = \sigma_{\varphi(0)} \circ \varphi$,
so that $\psi(0) = 0$ and $C_\psi$ is a contraction $H^p(X) \to H^p(X)$. 
Since $\sigma_{\varphi(0)}^{-1} = \sigma_{\varphi(0)}$, we get that 
$C_\varphi = C_\psi \circ C_{\sigma_{\varphi(0)}}$ is bounded on 
$H^p(X)$ once we have checked that
$C_{\phi}$ is bounded on $H^p(X)$ for any M\"obius transformation
$\phi$. This can be verified by the  change of variables 
$w = \phi(z)$ inside the integral
\[
\int_{\mathbb T}  \Vert f(r\phi(\xi))\Vert^p_X dm(\xi)
\]
defining $\Vert C_\phi(f(r\cdot))\Vert_{H^p(X)}$ for $0 < r < 1$ and letting $r \to 1$.

The minor point of difference between the above scalar- 
and vector-valued arguments for the Hardy spaces
relates to the potential absence of radial limits. 
In fact, it is well known that any $f \in H^p$  has a.e.\ radial limits 
$f(\xi) = \lim_{r \to 1^{-}} f(r\xi)$ on $\mathbb T$,
but this is not always true for functions in $H^p(X)$: 
the bounded analytic function $f\colon \mathbb D \to c_0$, where
\[
f(z) = (z^n), \quad z \in \mathbb D,
\]
 does not have radial limits anywhere on $\mathbb T$. 
In fact, the existence of a.e.\ radial limits for
any $f \in H^p(X)$, and any fixed $1 \le p \le \infty$, characterizes the analytic
Radon-Nikod\'ym property (ARNP) of the complex Banach space $X$.
The ARNP and the above result by Bukhvalov and Danilevich (1982)
is not needed here,  but  the reader may keep in mind that
e.g.\ every reflexive Banach space has the ARNP. See, e.g., \cite[p.\ 723]{LaT06} for references and a discussion of the ARNP.

Let $X$ be an infinite-dimensional Banach space. 
According to  Corollary \ref{qcor}.(2)  there are no  compact compositions 
$\widetilde{C_\varphi}\colon H^p(X) \to H^p(X)$.
This raises the general question of which are the relevant qualitative properties 
for composition operators on vector-valued spaces such as $H^p(X)$.
In \cite{LST98} the authors considered weak compactness, and related properties, 
for which there are satisfactory results.

Let $X$ and $Y$ be Banach spaces. Recall that the bounded linear operator
$U\colon X \to Y$ is \textit{weakly compact} if  there is a weakly convergent 
subsequence $(Ux_{n_{k}})$ for any bounded sequence $(x_n) \subset X$.
If $X$ and $Y$ are non-reflexive spaces, then weakly compact
$U\colon X \to Y$ are relatively small operators.

A  fundamental result of Shapiro \cite{Sh87} (see also \cite{Sh93} and \cite{CMC95})
says that for $1 \le p < \infty$ 
the composition operator $C_\varphi$ is compact $H^p \to H^p$ if and only if  
\begin{equation}\label{Sh}
\lim_{\vert w \vert\to 1} \frac{N(\varphi,w)}{\log(1/\vert w\vert)} = 0.
\end{equation}
Above $N(\varphi,w) = \sum_{z \in \varphi^{-1}(w)} \log(1/\vert z\vert)$, 
where $w \in \mathbb D \setminus \{\varphi(0)\}$,
is the Nevan\-lin\-na counting function of $\varphi$. 
Several other equivalent criteria for the compactness of $C_\varphi\colon H^p \to H^p$
are known in the literature, but (\ref{Sh}) suffices for our purposes.
Littlewood's inequality implies that
$N(\varphi,w) \le C \cdot \log(1/\vert w\vert)$ as $\vert w \vert \to 1$
for  some constant $C = C(\varphi)$ for any analytic map $\varphi\colon \mathbb{D} \to \mathbb{D}$,
see \cite[10.4]{Sh93}.
Shapiro's condition (\ref{Sh}) is interpreted as a little-oh condition 
describing the rate of decrease of the affinity of  $\varphi$ 
for the values $w$ as $\vert w \vert \to 1$. 

There is a precise connection between
the weak compactness of  $\widetilde{C_\varphi}$ on $H^1(X)$ and the
compactness of  $C_\varphi$ on  $H^1$.
Note that Corollary \ref{qcor}.(3) implies that  $X$
is reflexive, that is, $I_X$ is weakly compact,
whenever $\widetilde{C_\varphi}$ is  weakly compact on $H^p(X)$.
Hence only  $p = 1$ or $p = \infty$ are interesting for weak compactness, since $H^p(X)$ 
is itself reflexive if $1 < p < \infty$ and $X$ is reflexive, because $H^p(X)$
is then a closed subspace of the reflexive space $L^p({\mathbb T},X)$.
The vector-valued part of the following result comes from \cite{LST98}.
 
\begin{theorem}\label{H1}
Let $X$ be a complex reflexive Banach space, and $\varphi\colon \mathbb D \to \mathbb D$
be an analytic map. Then the following conditions are equivalent.
\begin{enumerate}
\item
$\widetilde{C_\varphi}\colon H^1(X) \to H^1(X)$ is weakly compact
\item
$C_\varphi\colon H^1 \to H^1$  is weakly compact
\item
$C_\varphi\colon H^1 \to H^1$  is compact
\item Shapiro's condition \eqref{Sh} holds.
\end{enumerate}
\end{theorem}

\begin{proof} 
The implication (1) $\Rightarrow$ (2) follows from Corollary \ref{qcor}.(3).
Sarason \cite{S92} proved that the weak compactness of
  $C_\varphi\colon H^1 \to H^1$  actually yields the compactness  of 
  $C_\varphi\colon H^1 \to H^1$, in other words that
  (2) $\Rightarrow$ (3). The equivalence of (3) and (4) is contained in 
  Shapiro's theorem. There remains to show that 
 $\widetilde{C_\varphi}$ is weakly compact
$H^1(X) \to H^1(X)$ whenever $\varphi$ satisfies Shapiro's condition.

We outline the proof  of the implication (4) $\Rightarrow$ (1).
The argument is based on a Littlewood-Paley type formula for 
$\Vert \widetilde{C_\varphi}(f)\Vert_{H^1(X)}$ derived from a formula of Stanton 
for continuous subharmonic maps. 
His formula \cite[Thm.\ 2]{St86} implies that
\begin{equation}\label{Sta}
\Vert f \circ \varphi\Vert_{H^1(X)} = \Vert f(0)\Vert_X + \frac{1}{2\pi}
 \int N(\varphi,w) d[\Delta(\Vert f\Vert_X)](w)
\end{equation}
for $ f \in H^1(X)$, where $d[\Delta(\Vert f\Vert_X)]$ denotes the distributional Laplacian
associated to the subharmonic map $z \mapsto \Vert f(z)\Vert_X$ on $\mathbb D$.

Define de la Vallee-Poussin operators $V_n$ for any $n \in \mathbb N$ by
\[
V_nf(z) = \sum_{k=0}^n \hat{f}_kz^k + \sum_{k=n+1}^{2n-1} \frac{2n-k}{n} \hat{f}_kz^k
\]
for analytic functions $f\colon \mathbb D \to X$ having the Fourier expansion
$f(z) = \sum_{k=0}^\infty \hat{f}_kz^k$. 
Then $(V_n)$ is a uniformly bounded sequence of operators on $H^1(X)$
and $V_n\colon H^1(X) \to H^1(X)$ are weakly compact for any $n$ if $X$ is a reflexive 
Banach space (in fact, $V_n$ factors through a finite direct sum of  copies 
of $X$). Moreover, given $\varepsilon > 0$ and 
$0 < r < 1$ there is $n_0 = n_0(\varepsilon,r) \in \mathbb N$ so that
\begin{equation}\label{dlVest}
\Vert f(z) - V_{n_0}f(z)\Vert_X \le \varepsilon \cdot \Vert f \Vert_{H^1(X)}
\end{equation}
holds for all $\vert z \vert \le r$ and $f \in H^1(X)$. 

If Shapiro's condition (\ref{Sh}) holds and $\varepsilon > 0$ is arbitrary, then there is 
$r \in (0,1)$ such that $N(\varphi,w) \le \varepsilon \cdot \log(1/\vert w\vert)$
for all $\vert w \vert > r$. Fix 
$n_0$ as in (\ref{dlVest}) corresponding to $\varepsilon$ and $r$. By applying
(\ref{Sta}) to $f - V_{n_{0}}f$ we get for $f \in H^1(X)$ that
\begin{align*}
\Vert \widetilde{C_\varphi}(f) & - \widetilde{C_\varphi}(V_{n_{0}}f)\Vert_{H^1(X)}  =
\frac{1}{2\pi} \int_{\{r < \vert z\vert < 1\}} N(\varphi,w) d[\Delta(\Vert f-V_{n_{0}}f\Vert_X)](w)\\
& + \frac{1}{2\pi} \int_{\{\vert z\vert \le r\}} N(\varphi,w) d[\Delta(\Vert f-V_{n_{0}}f\Vert_X)](w) \equiv I_1+ I_2.
\end{align*}
The choice of $r \in (0,1)$ and (\ref{Sta}) applied to $\psi(z) = z$ give 
\begin{align*}
I_1 & \le \frac{\varepsilon}{2\pi} \int_{\mathbb D} \log(1/\vert w\vert) d[\Delta(\Vert f-V_{n_{0}}f\Vert_X)](w) = \varepsilon \Vert f - V_{n_{0}}f\Vert_{H^1(X)} \\
& \le C \cdot \varepsilon \Vert f\Vert_{H^1(X)}
\end{align*}
for a uniform constant $C$. Moreover, it can be shown that 
$I_2 \le 4\varepsilon \Vert f\Vert_{H^1(X)}$ by using the estimates $N(\varphi,w) 
\le  \log(1/\vert w\vert)$ and (\ref{dlVest}). For this it is convenient to introduce
a cut-off function $\psi \in C^\infty_0(\mathbb D)$ satisfying $0 \le \psi \le 1$, 
$\psi = 1$ on $\{z\colon \vert z \vert \le r\}$ 
and $\psi = 0$ on $\{z \in \mathbb D\colon \vert z\vert \ge (1+r)/2\}$. 
We refer to \cite[Prop.\ 2 and Thm.\ 3]{LST98} for the complete technical details.

Thus $\Vert \widetilde{C_\varphi} - \widetilde{C_\varphi}V_{n_{0}}\Vert \le C' \cdot \varepsilon$, 
where $C'$ does not depend on $\varepsilon$, so that 
$\widetilde{C_\varphi}$ is well approximated by the weakly compact
operators $\widetilde{C_\varphi}V_{n_{0}}$ for suitable $n_0$. This means that 
$\widetilde{C_\varphi}$ is weakly compact $H^1(X) \to H^1(X)$.
\end{proof}

Theorem \ref{H1} corresponds to the following  template for many results 
about weak compactness, 
as well as  other qualitative properties, of analytic composition operators 
on vector-valued spaces. 

\begin{proposition}\label{gwc}
 Let $A$ be a Banach space of analytic functions on $\mathbb D$ and  
$A(X)$ a vector-valued version of $A$, such that $(A,A(X))$ satisfies 
(AF1)--(AF4). Suppose that $X$ is a complex reflexive Banach space 
and $\varphi\colon \mathbb D \to \mathbb D$ is an analytic map, so that
$\widetilde{C_\varphi}$ is bounded $A(X) \to A(X)$. 
Assume moreover that the following conditions hold:

\begin{itemize}
\item[(C1)] if $C_\varphi$ is weakly compact  $A \to A$, then $C_\varphi$ is compact  $A \to A$,  and 

\item[(C2)] if $C_\varphi$ is compact  $A \to A$, then the vector-valued composition
$\widetilde{C_\varphi}$ is weakly compact  $A(X) \to A(X)$.
\end{itemize}

Then one has the characterization

\begin{itemize}
\item[(C)] $\widetilde{C_\varphi}$ is weakly compact 
$A(X) \to A(X)$ $\Leftrightarrow$ 
 $C_\varphi$ is compact  $A \to A$.
\end{itemize}
\end{proposition}

We stress that the above general scheme is only a guiding principle and in practice
the techniques for establishing (C2)  depend on 
$A$ and its vector-valued extension $A(X)$. Moreover, the criteria 
for the compactness of the operator $C_\varphi\colon A \to A$ usually depend on $A$. 
It is straightforward to modify the scheme of Proposition \ref{gwc} to apply
 to vector-valued compositions $\widetilde{C_\varphi}\colon A(X) \to B(X)$ 
 between different spaces,
where $(A,A(X))$ and $(B,B(X))$ satisfy the properties (AF1)--(AF4).

Condition  (C1) is a problem of independent interest for composition operators
$A \to A$.
Recently Lefevre, Li, Queffelec and Rodriguez-Piazza
\cite{LLQR13} constructed the first example of a Banach space $A$ 
of complex-valued analytic functions on $\mathbb D$,
 where (C1) fails  for some symbol $\varphi$, see Example \ref{HO} below.
 
We next look at cases where Proposition \ref{gwc} apply.
Let $v_\alpha(z) = (1 - \vert z\vert^2)^\alpha$ 
for $z \in \mathbb D$ and $\alpha > -1$. 
The analytic function $f\colon \mathbb D \to X$ belongs to the weighted Bergman space 
$A^\alpha_p(X)$ if
\[
\Vert f\Vert_{A^\alpha_p(X)}^p = \int_{\mathbb D} \Vert f(z)\Vert_X^p v_\alpha(z) dA(z) <
\infty,
\]
where $dA$ is the area Lebesgue measure normalized by $A(\mathbb D) = 1$
and $1 \le p < \infty$.
The classical Bergman space $A^p(X)$ is obtained for $\alpha = 0$.
The following result was established in \cite{BDL01}, but the special 
case $A^1(X)$ was already  contained in \cite{LST98}.

\begin{theorem}\label{Berg}
Let $X$ be a complex reflexive Banach space, $\varphi\colon \mathbb D \to \mathbb D$
an analytic map and $\alpha > -1$. Then the following conditions are equivalent.

\begin{enumerate}
\item
$\widetilde{C_\varphi}\colon A^\alpha_1(X) \to A^\alpha_1(X)$ is weakly compact
\item
$C_\varphi\colon A^\alpha_1 \to A^\alpha_1$  is compact
\item
$\varphi$ satisfies the condition
\[
\limsup_{\vert w\vert \to 1} \frac{N_{\alpha+2}(\varphi,w)}{(\log (1/\vert w\vert))^{\alpha +2}} = 0.
\]
\end{enumerate}
\end{theorem}

Above $N_\beta(\varphi,\cdot)$ is the generalized Nevanlinna counting function defined 
for $\beta > 0$ by 
\[
N_\beta(\varphi,w) = \sum_{z \in \varphi^{-1}(w)} (\log(1/\vert z\vert)^\beta,
\quad w \in D \setminus \{\varphi(0)\},
\]
so that $N(\varphi,\cdot) = N_1(\varphi,\cdot)$.  
Actually, \cite[Thm.\  8]{BDL01} contains the estimate
\[
dist(\widetilde{C_\varphi}, W(A^\alpha_1(X))) \le C \cdot  
\limsup_{\vert w\vert \to 1} \frac{N_{\alpha+2}(\varphi,w)}{(\log (1/\vert w\vert))^{\alpha +2}},
\]
where $W(A^\alpha_1(X))$ denotes the linear subspace consisting of the weakly com\-pact operators  $A^\alpha_1(X) \to A^\alpha_1(X)$ and $C$ is an absolute constant.

We list some further Banach spaces $A(X)$ for  which the characterization 
(C) for weak com\-pactness of composition operators are known to hold. 
We emphasize that the arguments establishing (C1) and (C2)  usually are specific for  $A(X)$,
and the relevant com\-pactness conditions for $C_\varphi$ depend on $A$. 
One verifies by inspection that these pairs $(A,A(X))$ satisfy (AF1)--(AF4).
 
\begin{itemize}
\item Let $v\colon \mathbb D \to (0,\infty)$ be a bounded continuous weight function, and
 $f\colon \mathbb D \to X$ be an analytic function. Recall that $f \in H^\infty_v(X)$ if 
\[
\Vert f\Vert_{H^\infty_v(X)} = \sup_{\vert z\vert < 1} v(z) \Vert f(z)\Vert_X < \infty.
\]
Let $H^\infty_v = H^\infty_v(\mathbb C)$. The case $v \equiv 1$ gives the classical
spaces $H^\infty(X)$ and $H^\infty$ of bounded analytic functions. It was shown
in  \cite{LST98} that (C) holds on $H^\infty(X)$ and this was extended to 
the case $H^\infty_v(X)$ by different means in \cite{BDL01}.

\item  The vector-valued Bloch space ${\mathcal B}(X)$ \cite{LST98}.
Recall that $f \in {\mathcal B}(X)$ if
\[
\sup_{z \in \mathbb D} (1-\vert z\vert^2) \Vert f'(z)\Vert_X < \infty.
\] 
See \cite{BDL01} for an alternative approach via weak spaces (section \ref{weak} below).

\item The space $CT(X)$ of vector-valued Cauchy transforms \cite{LaT06}. 
The argument proceeds via composition operators on
the vector-valued harmonic Hardy space $h^1(X)$. The com\-pactness criterion 
for $CT$ is due to Bourdon, Cima and Matheson. 

\item Vector-valued $BMOA(X)$-spaces, see section \ref{BMOAs}.
 
\item Weak vector-valued versions  of the above spaces, see section \ref{weak}.  
\end{itemize}

A modified general scheme as in Proposition \ref{gwc} also applies 
to other operator ideal properties, namely, just replace weak com\-pactness 
by the relevant ideal property in (C1) and (C2). We state two results of this kind
for  $H^1(X)$, respectively $A^1_\alpha(X)$,
from  \cite[Thm.\ 7]{LST98} and  \cite[Cor.\ 9]{BDL01}.
The operator $U\colon X \to Y$ is called \textit{weakly conditionally com\-pact} 
if $(Ux_n)$ has a weak
Cauchy subsequence $(Ux_{n_{k}})$ for any bounded sequence $(x_n) \subset X$.
Recall that by Rosenthal's $\ell^1$-theorem, see \cite[2.e.5]{LT77},  
$I_X$ is weakly conditionally com\-pact if
and only if $X$ that does not contain any subspaces linearly isomorphic
to $\ell^1$. By Proposition \ref{GF} this is the relevant class of spaces here.

\begin{theorem}\label{wcc}
Suppose that the Banach space $X$ does not contain any sub\-spaces linearly isomorphic
to $\ell^1$, and $\varphi\colon \mathbb D \to \mathbb D$ is
an analytic map. Let $A = H^1$ or $A = A^1_\alpha$ for $\alpha > -1$. Then 
$\widetilde{C_\varphi}$ is weakly conditionally com\-pact 
$A(X) \to A(X)$ if and only if 
 $C_\varphi$ is com\-pact  $A \to A$.

\end{theorem}

We mention for completeness that the cases $dim(X) < \infty$ are similar to the scalar case.

\begin{proposition}\label{comp}
Suppose that $dim(X) < \infty$, $(A,A(X))$ satisfies (AF1)--(AF4), 
and $\varphi\colon \mathbb D \to \mathbb D$ is
an analytic map. Then $\widetilde{C_\varphi}$ is com\-pact 
$A(X) \to A(X)$ if and only if  $C_\varphi$ is com\-pact  $A \to A$.
\end{proposition}

\begin{proof}
Let  $n = dim(X)$ and fix a biorthogonal system $\{(x_r,x^*_s)\colon 1 \le r, s \le n\}$ for $X$, so that
$x = \sum_{k=1}^n x_k^*(x)x_k$ for $x \in X$. Hence any $f \in A(X)$ 
can be written as  $f(z) = \sum_{k=1}^n f_k(z)x_k$, where 
$f_k = x_k^* \circ f \in A$ for $k = 1,\ldots,n$. Consequently, 
if  $C_\varphi$ is com\-pact $A \to A$, then 
$
f \mapsto \widetilde{C_\varphi}(f) = \sum_{k=1}^n C_\varphi (f_k) x_k
$
is com\-pact $A(X) \to A(X)$. 
For the converse note that Section \ref{frame} applies to 
this setting. 
\end{proof}

\subsection*{Other results}
Hornor and Jamison  \cite{HJ99} characterized the isometrically 
equivalent compositions $\widetilde{C_\varphi}$ and $\widetilde{C_\psi}$ 
on $H^p(X)$, respectively $S^p(X)$, 
for $p \neq 2$ and $X$ a Hilbert space. Here $f \in S^p(X)$ if the derivative $f' \in H^p(X)$.
Sharma and Bhanu \cite{SB99} studied e.g.\ normal and  unitary
compositions on $H^2(X)$, where $X$ is a Hilbert space.
Bonet and Friz \cite{BF02} characterized the weakly com\-pact compositions 
on weighted vector-valued  locally convex spaces 
of analytic functions on $\mathbb D$.
Composition operators on the Hilbert space-valued Fock space of entire 
functions on $\mathbb C$ were considered by  Ueki \cite{U11}.
See also \cite{Wa05} for results on the vector-valued Nevanlinna class.

\section{Vector-valued $BMOA$-spaces}\label{BMOAs}

In this section we discuss in more detail the case of 
composition operators on vector-valued $BMOA$ spaces, 
since there are several natural  vector-valued versions of $BMOA$,
and condition (C1) is a problem of independent interest.

Recall that the analytic function $f\colon \mathbb D \to \mathbb C$ belongs to $BMOA$, the space of analytic functions of bounded mean oscillation, if 
\[
\vert f\vert_{*} = \sup_{a \in \mathbb D} \Vert f \circ \sigma_a - f(a)\Vert_{H^2} < \infty,
\]
where $\sigma_a(z) = (a-z)/(1-\overline{a}z)$ for $z \in \mathbb D$.
The Banach space $BMOA$  is equipped with the norm
$\Vert f\Vert_{BMOA} = \vert f(0)\vert + \vert f\vert_{*}$.
$BMOA$ is often considered as a M\"obius-invariant version of $H^2$,
but its  Banach space structure  is  comp\-li\-cated, see e.g.\ \cite{Mu}. 
Recall also that $(H^1)^* \approx BMOA$

There are by now several equivalent characterizations of com\-pact
compositions $C_\varphi\colon BMOA \to BMOA$, see 
\cite{LNST13} for a list. The following double criterion due to W. Smith \cite{Sm99}
is the most relevant one for our purposes: 

\smallskip

\noindent \textit{Let $\varphi\colon \mathbb D \to \mathbb D$ be an analytic
map. Then $C_\varphi$ is com\-pact $BMOA \to BMOA$ if and only if 
\begin{gather}
\tag{S1}
   \lim_{\vert \varphi(a)\vert \to 1} \sup_{0 < \vert w\vert  < 1} \vert w\vert^2 
   N(\sigma_{\varphi(a)} \circ \varphi \circ \sigma_a,w) = 0,  \\
\tag{S2}
   \lim_{t\to 1} \sup_{\{a\colon \vert \varphi(a)\vert \leq R\}} 
   m \bigl( \{\zeta \in \partial \mathbb D\colon
     \vert (\varphi \circ \sigma_a)(\zeta)\vert > t\} \bigr) = 0.
\end{gather}
Subsequently it was observed in \cite{La09}  that (S1) can be  restated as 
\begin{equation} \tag{L}\label{L}
   \lim_{\vert \varphi(a)\vert  \to 1}
   \Vert \sigma_{\varphi(a)}\circ\varphi\circ\sigma_a\Vert_{H^2} = 0.
\end{equation}
}

Let $f\colon \mathbb D \to X$ be an analytic function. We say that  $f \in BMOA(X)$ if
\[
\vert f\vert_{*,X} = \sup_{a \in \mathbb D} \Vert f \circ \sigma_a - f(a)\Vert_{H^2(X)} 
< \infty,
\]
and let $\Vert f\Vert_{BMOA(X)} = \Vert f(0)\Vert_X + \vert f\vert_{*,X}$. There are also
other natural possibilities. By departing from a well-known
characterization of $BMOA$ in terms of Carleson measures, see \cite[Thm.\ VI.3.4]{G81}, 
let   $f \in BMOA_{\mathcal C}(X)$ if
\[
\vert f\vert_{{\mathcal C},X} = \sup_{a \in \mathbb D} 
\int_{\mathbb D} \Vert f'(z)\Vert_X^2(1-\vert \sigma_a(z)\vert^2)dA(z) < \infty.
\]
The norm in $BMOA_{\mathcal C}(X)$ is 
$\Vert f\Vert_{BMOA_{\mathcal C}(X)} = \Vert f(0)\Vert_X + \vert f\vert_{{\mathcal C},X}$.
Blasco \cite{B00} showed that $BMOA(X) = BMOA_{\mathcal C}(X)$, with 
equivalent norms, if and only if $X$ is linearly isomorphic to a Hilbert space.
Thus $BMOA(X)$ and $BMOA_{\mathcal C}(X)$ are different vector-valued 
versions of $BMOA$. (In section \ref{weak} we will meet yet another vector-valued
version of $BMOA$.)

Laitila \cite{La05}, \cite{La07} initiated the study of composition operators
on vector-valued $BMOA$-spaces. He observed that $\widetilde{C_\varphi}$ 
is bounded $BMOA(X) \to BMOA(X)$ and
$BMOA_{\mathcal C}(X) \to BMOA_{\mathcal C}(X)$ for any self-map $\varphi\colon
\mathbb D \to \mathbb D$.
Moreover, if $X$ is a reflexive Banach space and $\varphi$ satisfies 
conditions (S1) and (S2), then 
$\widetilde{C_\varphi}$ is weakly com\-pact both $BMOA(X) \to BMOA(X)$ and
$BMOA_{\mathcal C}(X) \to BMOA_{\mathcal C}(X)$. 
In order to  obtain a complete characterization following  Proposition \ref{gwc} one 
has to verify condition (C1) for $BMOA$. 
This was actually a problem stated by Tjani in her Ph.D.\ thesis
\cite{Tj96} and Bourdon, Cima and Matheson \cite{BCM99},
which was eventually solved in \cite{LNST13} as follows.

\begin{theorem}\label{BMOA}
The following conditions are equivalent for $\varphi\colon \mathbb D \to \mathbb D$:
\begin{enumerate}
\item
$C_\varphi\colon BMOA \to BMOA$ is com\-pact
\item
$C_\varphi\colon BMOA \to BMOA$ is weakly com\-pact
\item
(S1) holds (alternatively, (L) holds)
\end{enumerate}
\end{theorem}

It is part of the solution that condition  (S2) is redundant in 
Smith's characterization above.
The  combination of Theorem \ref{BMOA} with  \cite{La05}, \cite{La07} 
comp\-le\-tes the following result for these vector-valued
$BMOA$-spaces.

\begin{theorem}\label{BMOA(X)}
Let  $X$ be a reflexive Banach space, and $\varphi\colon \mathbb D \to \mathbb D$
an analytic function. Then the following conditions are equivalent.
\begin{enumerate}
\item
$\widetilde{C_\varphi}$ is weakly com\-pact $BMOA(X) \to BMOA(X)$
\item
$\widetilde{C_\varphi}$ is weakly com\-pact 
$BMOA_{\mathcal C}(X) \to BMOA_{\mathcal C}(X)$
\item
$C_\varphi\colon BMOA \to BMOA$ is com\-pact, that is, condition (S1) holds.
\end{enumerate}
\end{theorem}

The argument for Theorem \ref{BMOA} is quite intricate.
It applies measure density ideas for the radial limits of $\varphi$ combined with
a criterion  due to Leibov (1986), respectively  M\"uller and Schechtman (1989), which 
allows to extract copies 
of the unit vector basis in $c_0$ from bounded sequences in the subspace
$VMOA$ of $BMOA$. We refer to \cite{LNST13} 
for the full technical details.

The results of  sections \ref{strongs} and \ref{BMOAs}
 might suggest  that the weak com\-pactness of
$C_\varphi\colon A \to A$ always implies its com\-pactness 
$A \to A$ for any  Banach space $A$ of analytic functions on 
$\mathbb D$. However, this  is not the case \cite[Thm. 4.1]{LLQR13}:

\begin{example}\label{HO}
Let $\varphi$ be the  lens map
\[
\varphi(z) = \frac{(1+z)^{1/2} - (1-z)^{1/2}}{(1+z)^{1/2} + (1-z)^{1/2}}, \quad z \in \mathbb D.
\]
Then there is an Orlicz function $\psi$ so that 
$C_\varphi$ is weakly com\-pact $H^\psi \to H^\psi$, but not com\-pact, where
$H^\psi$ is the non-reflexive Hardy-Orlicz  space of analytic functions of $\mathbb D$
defined by $\psi$. 
\end{example}

\begin{question}
Characterize the weakly com\-pact compositions on the space $H^\psi(X)$ above.   
The operator $C_\varphi$ in Example \ref{HO}  factors  through $H^4$ by construction, 
and $\widetilde{C_\varphi}$ through the
reflexive space $H^4(X)$ for reflexive spaces $X$. 
Thus $\widetilde{C_\varphi}$  is weakly com\-pact 
$H^\psi(X) \to H^\psi(X)$,
so that (C) cannot hold for  $H^\psi(X)$. 
\end{question}

To the best of our knowledge the first example of a weakly compact analytic
composition operator which is not compact was obtained in the context
of uniform algebras defined on infinite-dimensional domains. Let $U_E$ be the
open unit ball of the Tsirelson space $E$ and $\varphi\colon U_E\to U_E$ the map $x\mapsto x/2$.
It was shown in \cite[Example 3]{AGL} that the composition operator $f\mapsto f\circ\varphi$ is
weakly compact, but non-compact, $H^\infty(U_E)\to H^\infty(U_E)$. Here $H^\infty(U_E)$ is
the uniform algebra of bounded scalar-valued analytic functions $U_E\to\mathbb C$.

\section{Weak vector-valued spaces}\label{weak}

Bonet, Domanski and Lindstr\"om \cite{BDL01} introduced the class of weak spaces of
vector-valued analytic functions into the study of vector-valued composition operators. 
One of their aims was to provide an alternative approach 
to \cite{LST98}, but  the weak spaces 
are in general different from the spaces considered in sections \ref{strongs} and \ref{BMOAs}.
On the other hand, for the class of weak spaces 
there are some general results concerning vector-valued compositions.

Suppose that  $E$ is a Banach space of analytic functions 
$f\colon \UnitDisk \to \mathbb C$ satisfying
the following conditions:

\begin{itemize}
\item[(W1)] $E$ contains the constant functions,

\item[(W2)] the closed unit ball $B_E$ is com\-pact in the com\-pact open topology $\tau_{co}$ 
of $\UnitDisk$. 
\end{itemize}

Recall that the vector-valued function $f\colon \mathbb D \to X$ is analytic if and only if
$x^* \circ f$ is analytic $\mathbb D \to \mathbb C$ for all $x^* \in X^*$. This fact suggests
to define $f \in wE(X)$ if 
\begin{align*}
\Vert f \Vert_{wE(X)}=\sup_{\Vert x^* \Vert_{X^*} \le 1}\Vert x^* \circ f \Vert_{E} < \infty. 
\end{align*}
By the closed graph theorem $\Vert f \Vert_{wE(X)}$ is finite if and only if 
$x^* \circ f \in E$ for all $x^* \in X^*$.
We will say that $wE(X)$ is the weak space of vector-valued analytic functions
$\mathbb D \to X$ modelled on $E$. 
The spaces appearing in sections  \ref{strongs} and \ref{BMOAs}, whose
norms involve pointwise norm quantities such as $\Vert f(z)\Vert_X$, 
will in the sequel be called strong spaces. Such a distinction between strong and 
weak spaces is not precise, since e.g.\ $wH^\infty(X) = H^\infty(X)$. 
If $E$ is a Banach space of harmonic functions on $\mathbb D$ 
which satisfies (W1) and (W2), then one may similarly define the weak space $wE(X)$
of vector-valued harmonic functions $\mathbb D \to X$, see \cite{LaT06}.
Weak type spaces first appeared  in the theory of
vector measures, see e.g. \cite[chap. 13]{Di67}.
The weak Hardy spaces $wH^p(X)$, and in particular their harmonic versions $wh^p(X)$, 
have been studied by Blasco \cite{JBlasco87}, as well as in  \cite{freniche:fatou, freniche:survey}. 
 
Weak spaces $wE(X)$ have a dual nature, since they also admit a canonical
isometrically isomorphic representation as certain spaces of bounded operators. 
This general fact was observed in  \cite{BDL01}. 
Note first that if $E$  satisfies (W1) and (W2), then

\begin{itemize}
\item[(W3)] the evaluation maps $\delta_z  \in E^*$ for  $z \in \mathbb D$, 
where $\delta_z(f) = f(z)$  for $f \in E$.
\end{itemize}

The Dixmier-Ng theorem \cite{JNg} implies that $E = V^*$, where
\[
V=\{u^*\in E^*\colon u^* \textrm{ is $\tau_{co}$-continuous on $B_E$}\}.
\] 
The identification of $f \in E$ with $u^* \mapsto u^*(f)$ gives the isometric isomorphism 
$E \to V^*$.
In addition, $V = [\delta_z \in E^*\colon z \in \mathbb D]$ by Hahn-Banach, 
where $[B]$ denotes the closed linear
span of  the subset  $B \subset E^*$. 
 We next formulate the general  linearization result from \cite{BDL01}, 
 which also implies that $wE(X)$ is a Banach space. An analogue
holds for  weak harmonic spaces,  see \cite{LaT06}.

\begin{theorem}\label{thm:linearization}
Suppose that $E$ satisfies (W1) and (W2), 
let $V = [\delta_z \in E^*\colon z \in \mathbb D]$ and $X$ be a complex Banach space. 
Then there is an isometric isomorphism $\chi \colon L(V,X)\to wE(X)$, so that 
\begin{equation*}
(\chi (T))(z) = T(\delta_z), \quad  (\chi^{-1}(f))(\delta_z) = f(z),
\end{equation*}
hold for $T \in L(V,X)$, $f\in wE(X)$ and $z\in\UnitDisk$.
\end{theorem}
 
Special cases and variants of this linearization result were known earlier. 
The closest precursor is  the general results of Mujica \cite{JMujica} that apply to 
the case $E = H^\infty$. An explicit operator representation was obtained by 
Blasco \cite{JBlasco87} 
for the weak harmonic spaces $wh^p(X)$, where $1\le p \le \infty$. 

The study of composition operators between
weak spaces of analytic functions was initiated by Bonet, Doma\'nski and Lindstr\"om \cite{BDL01}, 
and this was extended in \cite{LaT06}  to weak spaces of vector-valued harmonic functions. 
Proposition 11 of \cite{BDL01} contains the result on weak com\-pactness stated below in Theorem \ref{thm:weakcpctnessweakspaces} for the weak  spaces, but
\cite{BDL01} only explicitly discusses the weighted $wH^\infty_v(X)$-spaces 
and the weak Bloch space $w{\mathcal B}(X)$ as examples. However, this approach  
applies to a large class of weak spaces of analytic functions, such as
the weak Hardy and weak Bergman spaces, as well as to $wBMOA(X)$, see the 
discussion below as well as  in \cite{LaT06, La05, La07}.

Let $\varphi\colon \mathbb D \to \mathbb D$ be an analytic map. 
The vector-valued composition 
$\widetilde{C_\varphi}$ is bounded $wE(X) \to wE(X)$ if and only if
$C_\varphi$ is bounded $E \to E$. In fact, if $x^* \in X^*$ then 
\begin{align*}
\Vert x^*\circ(\widetilde{C_\varphi} f)\Vert_{E}=\Vert C_\varphi(x^*\circ f)\Vert_{E}
\le \Vert C_\varphi \Vert \cdot \Vert x^*\circ f\Vert_{E},
\end{align*}
so that $\Vert \widetilde{C_\varphi}\Vert \le \Vert C_\varphi \Vert$.
For the converse it is worthwhile to point out that the framework from section
\ref{frame} applies to the weak spaces.

\begin{lemma}\label{waf}
If $E$ satisfies (W1) and (W2),  then the pair $(E,wE(X))$
satisfies (AF1)--(AF4) for any Banach space $X$
\end{lemma}

\begin{proof}
Conditions (AF1)--(AF3) are obvious. Towards (AF4) note that
\[
x^*(\widetilde{\delta_z}(f)) = \delta_z(x^* \circ f), \quad f \in wE(X),\   x^* \in X^*, 
\]
where we momentarily use $\widetilde{\delta_z}$ for the vector-valued evaluations
$f \mapsto f(z)$ taking $wE(X)$ to $X$.
\end{proof}

We stress that the following basic weak com\-pactness result from {\cite{BDL01} for
vector-valued compositions holds on  all weak spaces $wE(X)$. 
The proof uses different tools compared to  the analytic arguments 
in Sections \ref{strongs} and \ref{BMOAs}. Recall again from Corollary \ref{qcor}  that 
$\widetilde{C_\varphi}$ is never com\-pact $wE(X) \to wE(X)$ 
whenever $X$ is infinite dimensional. 

\begin{theorem}\label{thm:weakcpctnessweakspaces}
Suppose that  $E$ is a Banach space of analytic functions on $\mathbb D$ 
that satisfy (W1) and (W2). Let $\varphi\colon \mathbb D \to \mathbb D$ be an analytic map and 
$X$  a reflexive Banach space. If $C_\varphi\colon E\to E$ is com\-pact,
then  $\widetilde{C_\varphi}$ is weakly com\-pact $wE(X) \to wE(X)$.
\end{theorem}

\begin{proof}
Assume that $C_\varphi\colon E \to E$ is com\-pact. Its adjoint  
$(C_\varphi)^*\colon E^* \to E^*$ satisfies
\[
(C_\varphi)^*(\delta_z)=\delta_{\varphi(z)}, \quad z\in\UnitDisk,
\]
so that  $(C_\varphi)^*(V)\subset V$. We obtain the factorization
$\widetilde{C_\varphi} = \chi\circ U_\varphi\circ  \chi^{-1}$, where $U_\varphi$ is the 
operator composition map 
\[
T\mapsto I_X\circ T\circ (C_\varphi)^*|_{V}; \quad L(V, X)\to L(V, X),
\]
and $\chi$ is   the isometric isomorphism $L(V,X) \to wE(X)$ from Theorem \ref{thm:linearization}.
Since $(C_\varphi)^*|_{V}$ is a com\-pact operator $V \to V$ by duality,
and $I_X$ is weakly com\-pact, it follows from a general result of Saksman and Tylli, 
see \cite[Prop.\ 2.3]{ST06}, that
the operator composition  $U_\varphi$ is weakly com\-pact $L(V,X) \to L(V,X)$. 
Consequently $\widetilde{C_\varphi}$ is weakly com\-pact $wE(X) \to wE(X)$. 
\end{proof}

Theorem \ref{thm:weakcpctnessweakspaces} verifies
condition (C2) from Proposition \ref{gwc} for the weak spaces $wE(X)$. 
The following observation includes many examples.

\begin{proposition}\label{gws}
Suppose that  $E$ is a Banach space of analytic functions on $\mathbb D$ that
satisfy (W1) and (W2), let $\varphi\colon \mathbb D \to \mathbb D$ be an analytic map so that 
$C_\varphi$ is bounded  $E \to E$ and 
$X$  a reflexive Banach space. Suppose moreover: 

\begin{itemize}

\item[(C1)] if $C_\varphi$ is weakly com\-pact  $E \to E$, then $C_\varphi$ is com\-pact  $E \to E$.  
\end{itemize}

Then one has the characterization

\begin{itemize}
\item[(C)] $\widetilde{C_\varphi}$ is weakly com\-pact 
$wE(X) \to wE(X)$ $\Leftrightarrow$ 
 $C_\varphi$ is com\-pact  $E \to E$.
\end{itemize} 

Moreover, (C1) holds e.g.\ if $E$ is one of the following spaces: 
$H^1, A^1_\alpha$ for $\alpha > -1$, $BMOA$, 
$H^\infty_v$, where $v$ is a bounded continuous weight on $\mathbb D$, or $\mathcal B$.
\end{proposition}  

The preceding examples cover results for $wH^\infty_v(X)$ and $w{\mathcal B}(X)$ from \cite{BDL01},  $wH^1(X)$  \cite{LaT06},  and  $wBMOA(X)$ (combine Theorem \ref{BMOA}
with \cite{La05, La07}).

The results of sections \ref{strongs} --  \ref{weak} raise the question of what is the precise connection 
between these strong and weak spaces of vector-valued analytic functions.
Clearly $wH^2(\ell^2) \approx L(\ell^2)$ by Theorem \ref{thm:linearization}, whereas 
$H^2(\ell^2)$ is a separable Hilbert space, so the difference can be huge.
On the other hand,  \cite{BDL01} observed that 
$wH^\infty_v(X) = H^\infty_v(X)$ (equal norms)
and  $w{\mathcal B}(X) \approx {\mathcal B}(X)$ (equivalent norms).
It is evident  that e.g.\  $H^p(X) \subset wH^p(X)$, and
\[
\Vert f \Vert_{wH^p(X)}\le \Vert f \Vert_{H^p(X)},\quad f \in H^p(X),
\]
where  $1 \le p < \infty$ and $X$ is any Banach space. 
Blasco \cite{JBlasco87} observed that 
$h^1(C(\UnitCircle )) \varsubsetneq wh^1(C(\UnitCircle))$ and 
$h^p(L^{p'}) \varsubsetneq wh^p(L^{p'})$ for $1 < p < \infty$. 
Subsequently Freniche, Garc\'\i a-V\'azquez and Rodr\'\i guez-Piazza
\cite{freniche:fatou, freniche:survey} exhibited functions $f \in wh^p(X) \setminus h^p(X)$ and $g \in wH^p(X) \setminus H^p(X)$ for $1 \le p < \infty$ and any $X$. 
Fairly concrete functions of this kind were provided in \cite{La05, LaT06},  
and \cite{LTY09} contains  the analogous results for the weak vs.\ strong  Bergman norms.
In fact, the norms 
\[
\Vert \cdot \Vert_{wH^p(X)} \nsim \Vert \cdot \Vert_{H^p(X)}
\]
are non-equivalent on $H^p(X)$.
Strict inclusions $BMOA(X) \varsubsetneq wBMOA(X)$ and 
$BMOA_\mathcal C(X) \varsubsetneq wBMOA(X)$ for any infinite-dimensional $X$
were obtained in \cite{La05, La07}. 
A common feature of these examples for arbitrary $X$ is the use of 
Dvoretzky's $\ell^n_2$-theorem to transfer from the Hilbert space setting to $X$.

The linearization from Theorem \ref{thm:weakcpctnessweakspaces} can also be used  for 
other purposes. The following result  from \cite{BDL01} concern
 weak conditional com\-pactness on the spaces $wE(X)$. 

\begin{theorem}\label{thm:wcc}
Suppose that  $E$ is a Banach space of analytic functions on $\mathbb D$ that
satisfy (W1) and (W2). Let $\varphi\colon \mathbb D \to \mathbb D$ be an analytic map and 
$X$  a  Banach space that does not contain any subspaces linearly isomorphic to
$\ell^1$. If $C_\varphi\colon E\to E$ is com\-pact,
then  $\widetilde{C_\varphi}$ is weakly conditionally com\-pact $wE(X) \to wE(X)$.
\end{theorem}

The proof is analogous to that of Theorem \ref{thm:weakcpctnessweakspaces}, 
but instead apply \cite{JLS99} to deduce the  weak conditional com\-pactness of the 
operator composition $U_\varphi$.

Since  $wH^2(\ell^2) \approx L(\ell^2)$ is non-reflexive 
one may also look for a characterization of weakly com\-pact 
$\widetilde{C_\varphi}\colon wH^2(\ell^2) \to wH^2(\ell^2)$.
Note that the following observation is not included in Proposition \ref{gws}
since $H^2$ is reflexive.

\begin{proposition}\label{wHardyex}
 $\widetilde{C_\varphi}$ is weakly com\-pact $wH^2(\ell^2) \to wH^2(\ell^2)$
 if and only if  $\varphi$ satisfies Shapiro's condition (\ref{Sh}).
\end{proposition}

\begin{proof}
In view of Theorem \ref{thm:weakcpctnessweakspaces} there remains to show that
the weak compactness of  $\widetilde{C_\varphi}\colon wH^2(\ell^2) \to wH^2(\ell^2)$
implies condition (\ref{Sh}). As in  the proof of 
Theorem \ref{thm:weakcpctnessweakspaces} let $U_\varphi$ 
be the operator composition map
\[
S \mapsto S \circ (C_\varphi)^*|_{V}, \quad  L(V,\ell^2) \to L(V,\ell^2),
\]
 where $V = [\delta_z: z \in \mathbb{D}] = H^2$. We get that 
 $U_\varphi = \chi^{-1} \circ \widetilde{C_\varphi} \circ \chi$ 
is weakly compact $L(H^2,\ell^2) \to L(H^2,\ell^2)$.
It is known, see e.g. \cite[Example 2.6]{ST06}, that for such operator compositions
this yields the compactness
of $(C_\varphi)^*|_{V}$ on $V$. Hence  $C_\varphi$ is compact $H^2 \to H^2$,
so that  (\ref{Sh}) holds.
\end{proof}

The corresponding picture for the general class $wE(X)$ is quite complicated 
for reflexive $E$, and remains open, since 
the spaces $wE(X)$ can also be reflexive. For instance, the weak Hardy spaces
$wH^2(\ell^p) \approx  L(H^2, \ell^p) =  K(H^2,\ell^p)$ are reflexive for $1 < p < 2$ 
by Pitt's theorem \cite[Prop.\ 2.c.3]{LT77} and \cite[Sect.\ 2, Cor.\ 2]{Kal74}. Here
 $K(X,Y)$ denotes the space of com\-pact operators $X \to Y$.

\section{Compositions from weak to strong spaces}

A different line of study  
concerns the mapping properties of composition operators from 
weak to strong  spaces of analytic functions on $\mathbb D$, such as $wH^p(X) \to H^p(X)$.
This line was initiated by Laitila, Tylli and Wang in \cite{LTY09} for the Hardy and Bergman spaces, and subsequently the approach has been extended to 
weighted Bergman and Dirichlet spaces by Wang  \cite{JWang07, JWang08, JWang11}.
The question which motivated \cite{LTY09} came from S.~Kaijser for $X=\ell^2$.
 Recall from section \ref{weak} that e.g.\ $wH^2(\ell^2) \approx L(\ell^2)$ while
$H^2(\ell^2)$ is a separable Hilbert space, so that  boundedness of a  composition operator
$wH^2(\ell^2) \to H^2(\ell^2)$ entails strong compression.

Somewhat  surprisingly, the boundedness of $C_\varphi\colon wH^p(X)\to H^p(X)$  for 
$2\le p < \infty$ is related to composition operators in the Hilbert-Schmidt class
on $H^2$. Recall from \cite{JST73, CMC95} that $C_\varphi$ 
is a Hilbert-Schmidt operator on $H^2$ precisely when
\begin{align*}
\Vert C_\varphi \Vert_{HS}^2 = \int_{\UnitCircle} \frac{1}{1-|\varphi(\zeta)|^2}dm(\zeta)<\infty.
\end{align*}
The following result is taken from \cite{LTY09}, which also contains a formally similar
result for the vector-valued Bergman spaces. Note that results of this type
 have no counterparts in the scalar-valued theory.
 
\begin{theorem}\label{thm:weaktostrong}
Let $X$ be any infinite-dimensional complex Banach space. 

\begin{enumerate}
\item If $\Vert C_\varphi \Vert_{HS}<\infty$, then $C_\varphi$ is bounded 
$wH^p(X)\to H^p(X)$ for any $p$ satisfying $1\le p < \infty$.

\item The norm $\Vert C_\varphi\colon wH^p(X)\to H^p(X)\Vert $ is equivalent 
to $\Vert C_\varphi \Vert_{HS}^{2/p}$ for $2 < p < \infty$.

\item $\Vert C_\varphi\colon wH^2(X)\to H^2(X)\Vert = \Vert C_\varphi \Vert_{HS}$.
\end{enumerate}
\end{theorem}

Parts (2) and (3) are obtained by explicit  computations for $X=\ell^2$.
The extension to  arbitrary Banach spaces $X$ is based on  Dvoretzky's 
$\ell^n_2$-theorem and coefficient  multiplier results  \cite{JDuren} 
corresponding to bounded operators
\[
\sum_k a_kz^k \mapsto \sum_k \lambda_ka_kz^k, \quad H^2 \to H^p,
\]
where $(\lambda_k)$ is a fixed sequence. By contrast to Theorem \ref{thm:weaktostrong}, 
$\widetilde{C_\varphi}$ is bounded $wBMOA(\ell^2) \to BMOA(\ell^2)$ if and only if 
the scalar-valued operator $C_\varphi$ is bounded $\mathcal B \to BMOA$, see  \cite[Example 4.1]{LTY09}, where  $\mathcal B$ is the Bloch space.

\begin{question}\label{q:weaktostrong}
(a) Does part (2) of Theorem \ref{thm:weaktostrong} extend to $1\le p < 2$? 
The corresponding coefficient multiplier theorems $H^2 \to H^p$ for $1\le p < 2$
are not  readily useful.

(b) Characterize the weakly com\-pact compositions $\widetilde{C_\varphi}$ from $wH^1(X)$ to $H^1(X)$ if  $X$ is reflexive.
 It is  possible to show that if 
\begin{align*}
\lim_{s\to 1}\sup_{0<r<1} \int_{|\varphi(r\zeta)|>s} \frac{1}{1-|\varphi(\zeta)|^2}dm(\zeta)=0, 
\end{align*}
then $\widetilde{C_\varphi}$ is weakly com\-pact from $wH^1(X)$ to $H^1(X)$ (details omitted).
Note that 
$\widetilde{C_\varphi}$ is never com\-pact  $wH^1(X) \to H^1(X)$ 
for infinite-dimensional $X$ by section \ref{frame}.
\end{question}

One may also consider the composition operators $f\colon\mapsto f \circ \varphi$ 
from strong to weak spaces, e.g. as acting $H^p(X) \to wH^p(X)$, but this case does not  
produce new qualitative phenomena. This follows from the factorization
\[
\xymatrix{
H^p(X) \ar[rr]^{\widehat{C_{\varphi}}} \ar[rd]^{\widetilde{C_{\varphi}}}  && wH^p(X)\\
& H^p(X) \ar[ru]^{J}}
\]
where  $\widehat{C_{\varphi}}$ denotes the 
composition operator acting $H^p(X) \to wH^p(X)$ and $J: H^p(X) \to wH^p(X)$ is the continuous inclusion.
Hence, any  $\widehat{C_{\varphi}}$  is bounded $H^p(X) \to wH^p(X)$, 
and for $p = 1$ and reflexive spaces $X$
one obtains that  $\widehat{C_{\varphi}}\colon H^1(X) \to wH^1(X)$ is weakly com\-pact if and only
if $\varphi$ satisfies (\ref{Sh}), that is, $C_\varphi\colon H^1 \to H^1$ is com\-pact. 
(For the "only if"\ -part note that section \ref{frame} applies here.)
 
\section{Operator-weighted composition operators}\label{ow}

In the final section we briefly discuss extensions of  weighted composition operators 
to the vector-valued setting. Let $\psi\colon \UnitDisk\to \mathbb C$ 
and $\varphi\colon \mathbb D \to \mathbb D$ be given analytic maps. 
The weighted composition operator 
\begin{align*}
W_{\psi,\varphi}\colon f \mapsto \psi \cdot (f\circ\varphi)
\end{align*}
defines a linear map $H(\mathbb D) \to H(\mathbb D)$, where
$H(\mathbb D)$ denotes the linear space of analytic functions $\mathbb D \to \mathbb C$. 
Clearly $W_{\psi,\varphi}=M_\psi \circ C_\varphi$, where 
$M_\psi$ is the pointwise multiplier defined by
$M_\psi f=\psi \cdot f$ on $H(\mathbb D)$. Thus 
$W_{1,\varphi} = C_\varphi$ and $W_{\psi,id} = M_\psi$.

Weighted composition operators $W_{\psi,\varphi}$ have been extensively studied  on 
a range of complex-valued analytic function spaces,
and characterizations of e.g.\ boundedness and com\-pactness  
are known for many classical spaces. The case of the weighted spaces 
$H_v^\infty$ was resolved by Contreras and Hern\'andez-D\'iaz  in
 \cite{JCH2000} and Montes-Rodr\'iguez in \cite{MR00}.
For $1 < p <\infty$ and $H^p$  there is a Carleson measure characterization
 in \cite{JCH2001}, and the analogous results for the Bergman space $A^p$
 are found in \cite{JCZ}. The case of $BMOA$ can be found in \cite{La09}.
 We refer to e.g.\ \cite{JGP2010}  and \cite{JGG2008}
for other types of function-theoretic conditions for the boundedness of $W_{\psi,\varphi}$ 
on $H^p$.
Moreover, by  \cite{JCH2001} all weakly com\-pact  
weighted compositions $W_{\psi,\varphi}\colon H^1 \to H^1$ are com\-pact.

Independently Manhas \cite{JManhas08} and the authors \cite{JLaT09} proposed
the following natural analogue of weighted composition operators in the vector-valued
setting. Let $\varphi\colon \UnitDisk \to \UnitDisk$ be an analytic self-map 
and $\psi\colon \mathbb D \to L(X,Y)$ an analytic operator-valued map, where $X$ and $Y$ are complex Banach spaces.
Here $L(X,Y)$ denotes the space of bounded linear operators $X \to Y$.
Define the operator-weighted composition operator $W_{\psi,\varphi}$ by
$f \mapsto \psi(f\circ\varphi)$, that is,
\begin{align*}
(W_{\psi,\varphi}(f))(z) = \psi(z)(f(\varphi(z)), \quad z \in \mathbb D,
\end{align*}
for analytic functions $f\colon \mathbb D \to X$. Note that $z \mapsto  \psi(z)(f(\varphi(z))$
is an analytic map $\mathbb D \to Y$, so that $W_{\psi,\varphi}$
is a linear map $H(\mathbb D,X) \to H(\mathbb D,Y)$. Again 
$W_{\psi,\varphi}= M_\psi \circ C_\varphi$, 
where  $M_\psi$ denotes the operator-valued pointwise multiplier defined by
$(M_\psi f)(z)=\psi(z)(f(z))$ from $H(\mathbb D,X)$ to $H(\mathbb D,X)$.
Thus the operator-weighted composition operators form a 
much larger class compared to its scalar-valued relative.
Operator-weighted compositions appear naturally: 
for a large class of Banach spaces $X$ 
all linear onto isometries $H^\infty(X) \to  H^\infty(X)$ have the form $W_{\psi,\varphi}$, 
where $\psi(z) \equiv U$ is a fixed onto isometry of $X$ and $\varphi$ is an automorphism
of $\mathbb D$,  see  \cite{JLin90} and  \cite{JCJ90}.
The above definition of  $W_{\psi,\varphi}$ is analogous to that of weighted compositions
on spaces $C(K,X)$ of continuous functions, see e.g.\ \cite{JR88},
where $K$ is a com\-pact Hausdorff space.

The present knowledge of operator-weighted compositions 
is fairly rudimentary and  most of the results deal with  weighted $H^\infty_v(X)$ spaces 
and  their locally convex variants.  
Characterizations of  boundedness and (weak) com\-pactness of $W_{\psi, \varphi}$ between  $H_v^\infty(X)$ spaces were obtained in \cite{JLaT09}, 
and results for certain locally convex spaces of analytic vector-valued functions
are found in \cite{JManhas08} and \cite{JBonetEA12}. 
Boundedness of  the operator multipliers $M_\psi$ in related settings has been considered e.g.\ 
 in \cite{JManhas03, JManhas08, JManhas11, JWang11}.
 
We next state the main results from \cite{JLaT09}, which are vector-valued extensions of scalar results from \cite{JCH2000} and  \cite{MR00}. For a bounded continuous weight function 
$v\colon \UnitDisk\to (0,\infty)$, put
\[
\tilde v(z)=(\sup\{|f(z)|\colon \Vert f\Vert_{H_v^\infty}\le 1\})^{-1}.
\]
If  $\psi\colon \mathbb D \to L(X,Y)$ is a given analytic operator-valued map define 
the auxiliary linear map $T_\psi$ by $x\mapsto \psi(\cdot)x$. 
It follows that  $T_\psi$ is bounded
$X\to H_w^\infty(Y)$ whenever $W_{\psi,\varphi}$ is bounded $H_v^\infty(X)\to H_w^\infty(Y)$.

\begin{theorem}\label{thm:wco}
(1) Let $v$ and $w$ be weight functions. Then
\begin{align*}
\Vert W_{\psi,\varphi} \colon H_v^\infty(X)\to H_w^\infty(Y) \Vert = \sup_{z\in\UnitDisk}\frac{w(z)}{\tilde{v}(\varphi(z))}\Vert \psi(z)\Vert_{L(X,Y)}.
\end{align*}

(2) Assume that $v$ and $w$ are radial weight functions. Then $W_{\psi,\varphi}$ is 
com\-pact (respectively, weakly com\-pact)  $H_v^\infty(X)\to H_w^\infty(Y)$ if and only if both
\begin{equation}\label{owc}
\limsup_{|\varphi(z)|\to 1}\frac{w(z)}{\tilde{v}(\varphi(z))}\Vert \psi(z)\Vert_{L(X,Y)}=0,
\end{equation}
and  $T_\psi$ is com\-pact (respectively, weakly com\-pact) $X\to H_w^\infty(Y)$.
\end{theorem}

Parts of Theorem \ref{thm:wco} were independently
obtained in  \cite{JManhas08} together with other results.
Clearly the case $X=Y$ and $\psi(z) \equiv I_X$
yields  the boundedness and weak com\-pactness results 
from \cite{LST98, BDL01} for $\widetilde{C_\varphi}$,
since $T_\psi$ is then (essentially) $I_X$. 

Theorem \ref{thm:wco} points to some striking differences between scalar- and 
vector-valued weighted compositions as well as 
between operator-weighted and standard composition operators. 
For instance, the auxiliary operators $T_\psi$
play no role for scalar weighted compositions. Moreover,
 $W_{\psi,\varphi}$ can easily be com\-pact  $H_v^\infty(X)\to H_w^\infty(Y)$ for 
 infinite-dimensional spaces $X$ and $Y$. 
 For instance, let $\Vert \varphi\Vert_\infty < 1$
 and $\psi(z) \equiv U \in K(X,Y)$,  whence (\ref{owc}) holds and $T_\psi$ is
 com\-pact $X\to H_w^\infty(Y)$.
 Operator-weighted compositions  $W_{\psi,\varphi}\colon H_v^\infty(X)\to H_w^\infty(X)$
 do not always factor through $I_X$ as in Proposition \ref{GF}, but (F2) is replaced by
the following factorization for any $z \in \mathbb D$: 
\[
\xymatrix{
H_v^\infty(X) \ar[r]^{W_{\psi,\varphi}}  & H_w^\infty(Y) \ar[d]^{\delta_{z}}\\
X \ar[u]^j \ar[r]^{\psi(z)} & Y}
\]

There are also examples where $W_{\psi,\varphi}$ is weakly com\-pact
$H_v^\infty(X)\to H_w^\infty(X)$ but not com\-pact. For this one may use the fact 
\cite[Thm.\ 4.4]{JLaT09} that 
 if $w$ is a radial weight and $\psi\in H_w^0(L(X,Y)))$, then $T_\psi$ is com\-pact 
 (respectively, weakly com\-pact) if and only if 
 \[
 \psi(\UnitDisk)\subset K(X,Y) \quad \textrm{(respectively}, \psi(\UnitDisk)\subset W(X,Y)).
 \]
  Here $H_w^0(L(X,Y)))$
 denotes the closure of the analytic $L(X,Y)$-valued polynomials
 in $H_w^\infty(L(X,Y)))$, and $W(X,Y)$ the weakly com\-pact operators $X \to Y$.
 There are further differences between $H^\infty_v(X)$ and the  locally convex setting as studied in  \cite{JBonetEA12}.

The following problem stated in \cite{JLaT09} appears quite challenging.

\begin{question}
Characterize boundedness and (weak) com\-pactness of 
operator-weighted compositions  $W_{\psi,\varphi}\colon H^p(X) \to H^p(Y)$ 
as well as $A^p(X) \to A^p(Y)$ for $1\le p <\infty$.
 The argument from \cite{JCH2001} based on 
 Carleson measure techniques does not readily extend to the vector-valued setting. 
\end{question}

\end{document}